\documentclass[]{article}

\usepackage{amsmath,amssymb,amsthm}
\usepackage{hyperref}
\usepackage{tikz}
\usepackage{algorithm}
\usepackage{algpseudocode}

\newcommand{\Z}{\mathbb{Z}}
\newcommand{\N}{\mathbb{N}}

\newcommand{\rank}{\mathrm{rk}}

\theoremstyle{plain}
\newtheorem{lemma}{Lemma}

\newtheorem{corollary}{Corollary}

\newtheorem{theorem}{Theorem}

\title{Cantor-Bendixson ranks of countable SFTs}
\author{Ilkka T\"orm\"a \\ University of Turku, Finland \\ LIRMM, Universit\'e Montpellier, France}

\begin{document}

\maketitle

\begin{abstract}
We show that the possible Cantor-Bendixson ranks of countable SFTs are exactly the finite ordinals and ordinals of the form $\lambda + 3$, where $\lambda$ is a computable ordinal.
This result was claimed by the author in his PhD dissertation, but the proof contains an error, which is fixed in this note.
\end{abstract}

\section{Introduction}

In his PhD dissertation \cite{To15}, the author studied the structure of countable multidimensional SFTs.
In Corollary 4.16, it was claimed that the Cantor-Bendixson ranks of countable SFTs are exactly the finite ordinals and the computable ordinals of the form $\lambda + n$, where $\lambda$ is a limit orinal and $n \geq 3$.
While the claim is correct, the proof contains an error, which we fix in this note.

\section{Definitions}

For completeness, we repeat the relevant definitions from \cite{To15}.
Let $\Sigma$ be a finite alphabet.
The set $\Sigma^*$ contains all finite words over $\Sigma$, and $v \prec w$ means that $v$ is a prefix of $w$.
The \emph{$d$-dimensional full shift} over $\Sigma$ is the set $\Sigma^{\Z^d}$ equipped with the product topology.
The group $\Z^d$ acts on $\Sigma^{\Z^d}$ by \emph{shifts}: $\tau_{\vec v}(x)_{\vec w} = x_{\vec v + w}$.
A \emph{pattern} is an element $P \in \Sigma^D$, where $D \subset \Z^d$ is a finite \emph{domain}.
The \emph{cylinder} of $P$ is the set $[P] = \{ x \in \Sigma^{\Z^d} \;|\; x|_D = P \}$.
Finite unions of cylinders are exactly the clopen subsets of $\Sigma^{\Z^d}$, and form a base of the topology.
For a word $w \in \Sigma^*$, $[w] = \{ x \in \Sigma^\N \;|\; w \prec x \}$.
A \emph{subshift} is a topologically closed and shift-invariant subset of $\Sigma^{\Z^d}$.
Alternatively, a subshift is defined by a set of \emph{forbidden patterns}: for every subshift $X$, there exists a set of patterns $\mathcal{P}$ with $X = \mathcal{X}_{\mathcal{P}} = \{ x \in \Sigma^{\Z^d} \;|\; \forall P \in \mathcal{P}, \vec v \in \Z^d : \tau_{\vec v}(x) \notin [P] \}$.
If $\mathcal{P}$ is finite, then $X$ is a \emph{shift of finite type}, or SFT for short.
Intuitively, SFTs are sets of configurations defined by bounded-range constraints.
A two-dimensional subshift $X$ has the \emph{bounded signal property} if its horizontal rows are contained in a one-dimensional countable SFT.
It is \emph{deterministic} if for all $x \in X$, the coordinate $x_{\vec 0}$ is determined by $x|_H$, where $H = \{ (i, j) \in \Z^2 \;|\; j > 0 \}$ is the upper half-plane.

The \emph{Cantor-Bendixson derivative} of a topological space $X$ is the set $X' = \{ x \in X \;|\; \text{$x$ is not isolated in $X$} \}$.
By transfinite iteration, we extend this notion to all ordinals:
\begin{itemize}
\item $X^{(0)} = X$,
\item $X^{(\alpha + 1)} = (X^{(\alpha)})'$,
\item $X^{(\lambda)} = \bigcap_{\beta < \lambda} X^{(\beta)}$ for limit ordinals $\lambda$.
\end{itemize}
It is easy to see that each $X^{(\alpha)}$ is a closed subset of $X$.
Since $X$ is a set and the sequence $(X^{(\alpha)})_\alpha$ is decreasing, the derivation process eventually terminates.
The \emph{Cantor-Bendixson rank} of $X$, denoted $\rank(X)$, is the least ordinal $\alpha$ with $X^{(\alpha+1)} = X^{(\alpha)}$.
The \emph{rank} of a point $x \in X$, denoted $\rank_X(x)$, is the least ordinal $\alpha$ with $x \notin X^{(\alpha+1)}$.
If $x \in X^{(\rank(X))}$, then $x$ has no rank in $X$.

A subset $X \subset \{0, 1\}^\N$ is \emph{effectively closed}, or $\Pi^0_1$, if there exists a computable set of words $W \subset \{0, 1\}^*$ with $X = \{0, 1\}^\N \setminus \bigcup_{w \in W} [w]$.
We extend these notions to $\Sigma^{\Z^d}$ in the standard way.

We define a certain type of nondeterministic finite state machine that can be simulated in a countable SFT.
An \emph{arithmetical program} is a quadruple $(Q, \delta, q_0, q_f)$, where $Q$ is a finite state set, $\delta \subset Q \times (\{{+}, {-}, {\cdot}, {/}\} \times \N \cup \N \cup \N^2 ) \times Q$ is a finite transition relation, $q_0 \in Q$ is the initial state and $q_f \in Q$ is the final state.
An \emph{instantaneous description}, or ID, of the machine is a pair $(q, n) \in Q \times \N$, representing an internal state and the value of a counter.
A \emph{transition} of the machine form this ID to another is denoted by $(q, n) \to (p, m)$.
Depending on the transition relation, only certain transitions are possible:
\begin{itemize}
\item If $(q, ({*}, m), p) \in \delta$ for some ${*} \in \{{+}, {-}, {\cdot}, {/}\}$, $m \in \N$ and $p \in Q$, then $(q, n) \to (p, n * m)$ is a valid transition.
\item If $(q, n, p) \in \delta$ for some $p \in Q$, then $(q, n) \to (p, n)$ is a valid transition.
\item If $(q, (m, k), p) \in \delta$ for some $m, k \in \N$ and $p \in Q$, and $n \equiv m \bmod k$, then $(q, n) \to (p, n)$ is a valid transition.
\end{itemize}
The machine is initialized in $(q_0, n)$ for some initial counter value $n \in \N$, and runs nondeterministically until it reaches the final state $q_f$, at which point it \emph{accepts} its input.
We may assume it never performs an inexact division or subtracts a number greater than the counter value during such a computation.
We say $M$ is \emph{reversible} if for each ID $(p, m)$ there exists at most one ID $(q, n)$ with $(q, n) \to (p, m)$.
A reversible arithmetical program can simulate an arbitrary Minsky machine (in the sense that there is a computable bijection between the sets of their computation histories) by storing the entire computation history into an extra counter and encoding counters $n_1, n_2, \ldots, n_k$ as the single value $p_1^{n_1} p_2^{n_2} \cdots p_k^{n_k}$, where the $p_i$ are distinct primes (Lemmas 2.4 and 2.5 in \cite{To15}).
On the other hand, nondeterministic Minsky machines are able to simulate arbitrary nondeterministic Turing machines.
Neither method of simulation introduces additional nondeterminism into the simulating machine.

\section{Results}

It is known that some ordinals cannot occur as the rank of a countable SFT.

\begin{lemma}
\label{lem:UpperBound}
Let $X \subset \Sigma^{\Z^d}$ be a countable SFT.
Then $\rank(X)$ is either finite or $\lambda + 3$ for some computable ordinal $\lambda$.
\end{lemma}

\begin{proof}
That $\rank(X)$ is a computable ordinal was claimed in \cite{CeClSmSoWa86}, where the authors referred to \cite{Kr59} for a proof.
By Theorem 5.3 of the pre-print \cite{BaJe13}, the rank of a countable SFT cannot have the form $\lambda + n$, where $\lambda$ is a limit ordinal and $n \in \{0, 1, 2\}$.
Note that Lemma 5.2 of the same pre-print, which is used in the proof of the above result, is technically incorrect, but not in a way that affects the Theorem (if the set of periods $\mathcal{P}$ used in the Lemma contains parallel vectors, its claim is not necessarily true, but in the proof of the Theorem, we can guarantee that all vectors are non-parallel).
\end{proof}

We use the following result to construct well-behaved countable $\Pi^0_1$ sets with specified ranks.
By a computably enumerable point $x \in \{0, 1\}^\N$, we mean the characteristic function of a computably enumerable set.

\begin{lemma}[Corollary of Theorem 1 in \cite{ChDo93}]
\label{lem:ExistsRE}
For each computable ordinal $\lambda \neq 0$ and computably enumerable uncomputable $x \in \{0, 1\}^\N$, there exists a countable $\Pi^0_1$ set $S \subset \{0, 1\}^\N$ with $S^{(\lambda)} = \{ x \}$.
\end{lemma}

The rest of this note is dedicated to the proof of the following theorem.

\begin{theorem}
\label{thm:main}
For each infinite computable ordinal $\lambda$, there exists a deterministic countable SFT $X \subset \Sigma^{\Z^2}$ with the bounded signal property and $\rank(X) = \lambda + 3$.
\end{theorem}

\begin{proof}
By Lemma~\ref{lem:ExistsRE}, there is a computably enumerable uncomputable point $x^{\mathrm{max}} \in \{0, 1\}^\N$ and a countable $\Pi^0_1$ set $S \subset \{0, 1\}^\N$ with $S^{(\lambda)} = \{ x^{\mathrm{max}} \}$.
The idea is to construct an arithmetical program that runs forever and nondeterministically produces an element of $S$ one bit at a time.
Incorrect choices are recognized in finite time and cause the program to halt.
We simulate the program in a simple countable SFT $X$, and since the choices are visible in the configurations, we obtain a computable and continuous embedding of $S$ into $X$.
However, $X$ also contains `degenerate' configurations that do not encode elements of $S$, increasing its Cantor-Bendixson rank.
Some increase is inevitable by Lemma~\ref{lem:UpperBound}, and we will optimize it by encoding the highest-rank element $x^{\mathrm{max}}$ as a degenerate configuration.
Since $x^{\mathrm{max}}$ is computably enumerable, the arithmetical program is able to distinguish it from other elements of $S$.
This last step was not present in the construction in \cite{To15}, resulting in the bound $\lambda + 4$, which was incorrectly reported as $\lambda + 3$.

We proceed with the construction.
Let $w^0, w^1, \ldots \in \{0, 1\}^*$ be a computably enumerable set of forbidden prefixes for $S$.
Let $T$ be a Turing machine such that $x^{\mathrm{max}}_n = 1$ if and only if $T$ halts on $n$.
Let $M = (Q, \delta, q_0, q_f)$ be a reversible arithmetical program corresponding to Algorithm~\ref{alg:Algo}.
We may assume that every other transition of $M$ involves multiplying the counter by 2, and the counter is never divided by an even number; this is analogous to giving a counter machine one extra counter that it increments on every other step.

\begin{algorithm}[t]

\begin{algorithmic}[1]
  \State $k \gets \mathrm{input}$ \Comment{Program parameter}
  \State $\mathcal{T} \gets \emptyset$ \Comment{A set of Turing machines}
  \State $v \gets \epsilon$ \Comment{A prefix of $x^{\mathrm{max}}$}
  \State $w \gets \epsilon$ \Comment{A prefix of $y$}
  \State $j \gets 0$ \Comment{The length of $w$}
  \For{$i \in \N$} \Comment{The length of $v$}
    \State \textbf{guess}~$v_i \in \{0, 1\}$
    \If{$v_i = 1$}
      \State \textbf{verify}~that $T(i)$ halts \Comment{This step may take forever} \label{ln:HaltCheck}
      \State \textbf{guess}~$w_j \in \{0, 1\}$
      \For{$p \in \{0, \ldots, j\}$}
        \State \textbf{verify}~$w^p \not\prec w$ \label{ln:SCheck}
      \EndFor
      \If{$j = k-1$}
        \State \textbf{verify}~$w \neq v_{[0, k-1]}$ \label{ln:kCheck}
      \EndIf
      \State $j \gets j+1$
    \Else
      \State $\mathcal{T} \gets \mathcal{T} \cup \{T(i)\}$
      \State \textbf{verify} that no $T(\ell) \in \mathcal{T}$ halts in $i$ steps \label{ln:HaltCheck2}
    \EndIf
  \EndFor
\end{algorithmic}

\caption{Take a number $k \in \N$, enumerate the configuration $x^{\mathrm{max}}$ and another configuration $y \in S \setminus [x^{\mathrm{max}}_{[0, k-1]}]$.}
\label{alg:Algo}
\end{algorithm}

On a step marked with \textbf{verify}, the program $M$ checks the condition and halts if it does not hold.
The steps markes with \textbf{guess} are nondeterministic choices.
If the algorithm does not halt and always passes the check on line~\ref{ln:HaltCheck}, it enumerates a configuration $x = \lim v \in \{0, 1\}^\N$, and if $x$ has infinitely many $1$s, another configuration $y = \lim w \in \{0, 1\}^\N$.

\begin{lemma}
\label{lem:AlgoBehavior}
Let $k \in \N$ be the input of Algorithm~\ref{alg:Algo}.
If the algorithm runs forever and always passes the check on line~\ref{ln:HaltCheck}, then $\lim v = x^{\mathrm{max}}$ and $\lim w \in S \setminus [x^{\mathrm{max}}_{[0, k-1]}]$.
All such choices of $\lim w$ are possible.
\end{lemma}

\begin{proof}
Denote $x = \lim v$ and $y = \lim w$.
Suppose $x_n \neq x^{\mathrm{max}}_n$ for some $n \in \N$.
If $x^{\mathrm{max}}_n = 0$, then $T(n)$ does not halt, so the check on line~\ref{ln:HaltCheck} never finishes.
If $x^{\mathrm{max}}_n = 1$, then $T(n)$ halts after some $i$ steps, and the check on line~\ref{ln:HaltCheck2} fails.
Both possibilities contradict our assumptions, and hence $x = x^{\mathrm{max}}$.
In particular, since $x^{\mathrm{max}}$ is uncomputable, $x$ has infinitely many 1s.

For all $p \in \N$ we eventually guarantee $w^p \not\prec y$ on line~\ref{ln:SCheck}, which implies $y \in S$.
We also verify $x^{\mathrm{max}}_{[0, k-1]} \not \prec y$ on line~\ref{ln:kCheck} (note that $v$ is always at least as long as $w$), so that $y \notin [x^{\mathrm{max}}_{[0, k-1]}]$.
Since $y$ is otherwise unconstrained, the claim follows.
\end{proof}

We now construct the countable SFT $X$, following the constructions in Section 3.3 of \cite{To15} with minor modifications.
The alphabet of $X$ is $\Sigma = \{0, \mathrm{Z}, \ell, r\} \cup (P_M \times Q) \cup P'_M$, where $P_M$ and $P'_M$ are auxiliary finite state sets that depend on $M$.
Let $Y \subset \Sigma^\Z$ be the orbit closure of configurations of the form ${}^\infty 0 \ell^m (p, q) r^n p' \mathrm{Z}^\infty$, ${}^\infty 0 \ell^m \mathrm{Z}^\infty$ and ${}^\infty 0 \mathrm{Z}^\infty$, where $m, n \geq 0$, $p \in P_M$, $p' \in P'_M$ and $q \in Q$.
We can define $Y$ using forbidden words of length $2$, so that it is a one-dimensional SFT.
We require that each horizontal row of $X$ comes from $Y$.
We also require that the border of the $0$-symbols is at the same position on each row.
The area to the right of this border is the \emph{computation zone}, and the symbols $(p, q)$ and $p'$ is called the \emph{left and right heads}.

There are special elements $p_0 \in P_M$ and $p'_0 \in P'_M$ such that a configuration of $Y$ with the form ${}^\infty 0 \ell^m (p_0, q) p'_0 \mathrm{Z}^\infty$ corresponds to the ID $(q, m)$ of $M$.
The configuration ${}^\infty 0 \mathrm{Z}^\infty$ represents a computation that has not yet started; above it, we can have either another identical configuration or ${}^\infty 0 \ell \mathrm{Z}^\infty$.
For $m \geq 1$, above the configuration ${}^\infty 0 \ell^m \mathrm{Z}^\infty$ we may have either ${}^\infty 0 \ell^{m+1} \mathrm{Z}^\infty$ or ${}^\infty 0 \ell^m (p_0, q_0) p'_0 \mathrm{Z}^\infty$, where $q_0 \in Q$ is the initial state of $M$.
The number $m$ encodes the input $k$ to Algorithm~\ref{alg:Algo}.
Since the program $M$ uses powers of primes to encode multiple counters, we can denote $m = 3^k \cdot N$ with $N$ not divisible by 3, and assume that at the start of its computation, $M$ checks that $N$ is not divisible by any other prime that it uses in its computations.

The computation of $M$ proceeds upward in the SFT $X$.
Subtraction or addition of a constant number can be performed in one step by moving both heads to the left or right, respectively.
To check if the counter has an exact value $n \in \N$, the heads only need to look $n$ cells to the left.
The remaining operations require auxiliary steps.
To check the remainder of the counter value modulo a number $n$, the left head travels to the $0$-border and back, keeping track of the distance modulo $n$.
To multiply the counter value by a number, both heads start moving at constant speeds, the left head bounces off the $0$-border, and the heads meet at a new position, which is a multiple of the previous position.
Division by a constant is done analogously.
See Section 3.3 of \cite{To15} for a detailed explanation, and Figure~\ref{fig:SFT} for a visualization.
In case of an illegal operation of subtracting a number larger than the counter value or performing an inexact division, which may in principle happen during an infinite computation with no starting point, or if any of the checks of Algorithm~\ref{alg:Algo} fails, a tiling error is produced.
As the operations are local, they can be implemented by local rules in the SFT $X$.
It is not hard to see that $X$ is downward deterministic (since $M$ is reversible) and has the bounded signal property.

\begin{figure}[htp]
\begin{center}
\begin{tikzpicture}[scale=0.4]

\fill [black!35] (2,-6) rectangle (25,30);

\foreach \x in {3,...,8}{
  \pgfmathsetmacro{\y}{\x-7}
  \pgfmathsetmacro{\yp}{\y+1}
  \fill [black!15] (2,\y) rectangle (\x,\yp);
}

\foreach \lc/\rc [count=\y from 2] in
    {8/8,
     7/8,6/8,5/8,4/8,3/8,2/8,3/8,4/8,5/8,6/8,7/8,
     8/8,
     5/9,2/10,5/11,8/12,11/13,
     14/14,
     11/11,
     8/12,5/13,2/14,5/15,8/16,11/17,14/18,17/19}{
  \pgfmathsetmacro{\yp}{\y+1}
  \pgfmathsetmacro{\yh}{\y+0.5}
  \pgfmathsetmacro{\lh}{\lc+0.5}
  \pgfmathsetmacro{\rh}{\rc+1.5}
  \pgfmathsetmacro{\lm}{\lc+1}
  \pgfmathsetmacro{\rm}{\rc+2}
  
  \fill [black!15] (2,\y) rectangle (\lm,\yp);
  \fill [black!25] (\lm,\y) rectangle (\rm,\yp);
  \node at (\lh,\yh) {L};
  \node at (\rh,\yh) {R};
}

\draw (0,-6) grid (25,30);

\end{tikzpicture}
\end{center}
\caption{A configuration of the SFT $X$. The two heads are denoted by L and R (their internal states are not shown), and the four shades of gray denote 0, $\ell$, $r$ and Z. The simulated machine is initialized to $m = 6$. It then checks the congruence class of the counter modulo some number, multiplies the counter by 2, decrements it by 3, and starts another multiplication by 2.}
\label{fig:SFT}
\end{figure}

Let $T \subset \N$ be the set of numbers that are not divisible by any of the primes that $M$ uses during its computation.
Let $k \in \N$, $N \in T$ and $\vec v \in \Z^2$, and let $X(k, N, \vec v) \subset X$ be the set of configurations where the computation is initialized at $\vec v$ with input $3^k \cdot N$.
Since $X(k, N, \vec v)$ is open in $X$, for each ordinal $\alpha$ we have
\begin{equation}
\label{eq:Deriv}
X^{(\alpha)} = Z_\alpha \cup \bigcup_{k, N, \vec v} X(k, N, \vec v)^{(\alpha)},
\end{equation}
where $Z_\alpha \subset X$ contains only configurations where the computation never starts.

We claim that each derivative $X(k, N, \vec v)'$ is homeomorphic to $S \setminus [x^{\mathrm{max}}_{[0, k-1]}]$.
Suppose that $z \in X(k, N, \vec v)$ contains only correct guesses at each step of the algorithm, so that the check on line~\ref{ln:HaltCheck} always succeeds after finitely many steps.
For $n \in \N$, let $z^n \in X(k, N, \vec v)$ be the configuration where the $n$th $1$-bit of $x^{\mathrm{max}}$ is guessed incorrectly as $0$, and preceding guesses are made as in $z$.
Then $z = \lim_n z^n$, so $z \in X(k, N, \vec v)'$.
Also, since $z$ enumerates two configurations $x^{\mathrm{max}}$ and $y \in S \setminus [x^{\mathrm{max}}_{[0, k-1]}]$ by Lemma~\ref{lem:AlgoBehavior}, we may choose $z$ as the homeomorphic image of $y$.
On the other hand, if $z \in X(k, N, \vec v)$ contains an incorrect guess, then it must guess a $1$-bit of $x^{\mathrm{max}}$ as $0$, since other incorrect guesses result in tiling errors.
Then the algorithm is stuck on line~\ref{ln:HaltCheck} forever and can make no further nondeterministic choices, so $z$ is an isolated point of $X(k, N, \vec v)$.

We now claim that $\rank(X) = \lambda + 3$.
Since $\lambda$ is infinite, by~\eqref{eq:Deriv} we have $X^{(\lambda)} = Z_\lambda \cup \bigcup_{k, N, \vec v} X(k, N, \vec v)^{(\lambda)} = Z_\lambda$, where $Z_\lambda$ consists of configurations $z \in X$ where computation does not start.
If $z$ contains the left border of the computation zone, then it cannot contain both left and right heads of the program $M$: since every other step of the program $M$ involves multiplying the counter by 2, this would imply that the computation has a starting point in $z$.
Thus $z$ may contain at most one back-and-forth sweep of the left head, or the south border of the $\ell$-region but no heads, and in both cases $z$ is isolated in $Z_\lambda$.
Similarly, if $z$ contains the right head, it may contain at most one sweep of the left head, and is isolated.
Such configurations exist in $Z_\lambda$, since in every configuration of $X$ with a valid computation, the left head visits the left border of the computation zone infinitely many times.
All other configurations of $Z_\lambda$ are periodic and constitute a finite number of shift orbits, and thus we have $\rank(Z_\lambda) = 3$.
This concludes the proof.
\end{proof}

\begin{corollary}
The Cantor-Bendixson ranks of countable two-dimensional SFTs are exactly the finite ordinals and the computable ordinals $\lambda + n$, where $\lambda$ is a limit ordinal and $n \geq 3$.
\end{corollary}

\begin{proof}
Arbitrary finite ranks for countable SFTs are witnessed by the subshifts over the alphabets $\{0, \ldots, n\}$ that are constant in the vertical direction and nondecreasing in the horizontal direction.
The computable ordinals $\lambda + n$, where $\lambda$ is a limit ordinal and $n \geq 3$, are handled by Theorem~\ref{thm:main}.
Lemma~\ref{lem:UpperBound} shows that these are the only possibilities.
\end{proof}

\section*{Acknowledgments}

The author is thankful to Ville Salo for finding the error in \cite{To15}.

\bibliographystyle{plain}
\bibliography{bib}{}

\begin{thebibliography}{1}

\bibitem{BaJe13}
Alexis {Ballier} and Emmanuel {Jeandel}.
\newblock {Structuring multi-dimensional subshifts}.
\newblock {\em ArXiv e-prints}, September 2013.

\bibitem{CeClSmSoWa86}
Douglas Cenzer, Peter Clote, Rick~L. Smith, Robert~I. Soare, and Stanley~S.
  Wainer.
\newblock Members of countable {$\Pi^0_1$} classes.
\newblock {\em Ann. Pure Appl. Logic}, 31(2-3):145--163, 1986.
\newblock Special issue: second Southeast Asian logic conference (Bangkok,
  1984).

\bibitem{ChDo93}
Peter Cholak and Rod Downey.
\newblock On the {C}antor-{B}endixon rank of recursively enumerable sets.
\newblock {\em J. Symbolic Logic}, 58(2):629--640, 1993.

\bibitem{Kr59}
G.~Kreisel.
\newblock Analysis of the {C}antor-{B}endixson theorem by means of the analytic
  hierarchy.
\newblock {\em Bull. Acad. Polon. Sci. S\'er. Sci. Math. Astr. Phys.},
  7:621--626. (unbound insert), 1959.

\bibitem{To15}
Ilkka T{\"o}rm{\"a}.
\newblock {\em Structural and Computational Existence Results for
  Multidimensional Subshifts}.
\newblock {PhD} dissertation, University of Turku, 2015.

\end{thebibliography}

\end{document}